\documentclass[12pt]{amsart}

\usepackage{amsmath,amssymb,txfonts}
\usepackage{amssymb}
\usepackage{amsxtra}
\usepackage{amsthm, color}
\usepackage{txfonts}
\usepackage{graphicx}
\usepackage{color}
\usepackage{times}

\usepackage[cp1252]{inputenc}

\usepackage{mathrsfs}
\usepackage{graphicx}

\usepackage[active]{srcltx}

\newcommand{\rmv}[1]{}

\numberwithin{equation}{section}

\newcommand{\HH}{\mathcal{H}}
\newcommand{\HK}{\mathcal{K}}

\newcommand{\HS}{\mathcal{S}}
\newcommand{\HB}{\mathcal{B}}

\newcommand{\D}{\mathbb{D}}

\newcommand{\C}{\mathbb{C}}

\newcommand{\T}{\mathbb{T}}

\newcommand{\Hol}{\text{ Hol}}

\theoremstyle{plain}
\newtheorem{theorem}{Theorem}[section]
\newtheorem{lemma}[theorem]{Lemma}
\newtheorem{remark}[theorem]{Remark}

\newtheorem{prop}[theorem]{Proposition}

\theoremstyle{definition}

\begin{document}
\thanks{S. Luo was supported by an AARMS postdoctoral fellowship, NSERC of Canada (\#20171864) and the NNSF of China (\#11701167).}
\date{}

\author{Shuaibing Luo}
\address{School of Mathematics, and Hunan Provincial Key Laboratory of Intelligent information processing and Applied Mathematics, Hunan University, Changsha, 410082, PR China}
\email{sluo@hnu.edu.cn}


\title{Corona theorem for the Dirichlet-type space}

\begin{abstract}
This paper utilizes Cauchy's transform and duality for the Dirichlet-type space $D(\mu)$ with positive superharmonic weight $U_\mu$ on the unit disk $\D$ to establish the corona theorem for the Dirichlet-type multiplier algebra $M\big(D(\mu)\big)$ that: if
\[\{f_1,...,f_n\}\subseteq M\big(D(\mu)\big)\quad\text{and}\quad \inf_{z\in\D}\sum_{j=1}^n|f_j(z)|>0
\]
then
\[
\exists\,\{g_1,...,g_n\}\subseteq M\big(D(\mu)\big)\quad\text{such\ that}\quad \sum_{j=1}^nf_jg_j=1,
\]
thereby generalizing Carleson's corona theorem for $M(H^2)=H^\infty$ in \cite{Ca62} and Xiao's corona theorem for $M(\mathscr{D})\subset H^\infty$ in \cite{Xi98} thanks to
\[
D(\mu)=\begin{cases} \text{Hardy\ space}\ H^2\quad &\text{as}\quad d\mu(z)=(1-|z|^2)\,dA(z)\ \ \forall\ z\in\D;\\
\text{Dirichlet\ space}\ \mathscr{D}\ &\text{as}\quad  d\mu(z)=|dz|\ \ \forall\ z\in\T=\partial{\D}.
\end{cases}
\]
\end{abstract}
\maketitle


\section{Introduction}\label{s1}

\subsection{Holomorphic space $D(\mu)$}\label{s11}

Given a finite positive Borel measure $d\mu$ on the closed unit disk $\overline{\D}$ which is the union of the unit disk $\D$ and its boundary $\T$, let
\begin{equation*}
\label{e11}
U_\mu(z) = \int_{\D}\bigg(\log\left|\frac{1-\overline{w}z}{z-w}\right|^2\bigg)\frac{d\mu(w)}{1-|w|^2}+\int_{\T}\frac{1-|z|^2}{|1-\overline{w}z|^2}d\mu(w)\ \ \forall\ \ z\in \D.
\end{equation*}
This function $U_\mu$ is a positive superharmonic function on $\D$ -i.e.-
\begin{equation*}
\label{e12}
\partial\overline{\partial}U_\mu(z)\le 0\ \ \forall\ \ z\in\D
\end{equation*}
holds in the sense of distributions. If $dA$ and $H^2$ stand for the normalized area measure and the Hardy space on $\D$ respectively, then $D(\mu)$ is the Dirichlet space with the weight $U_\mu$, i.e., the class of all functions $f\in \text{Hol}(\D)$ satisfying
$$
\int_\D |f'(z)|^2 U_\mu(z)\,dA(z)<\infty.
$$
We have
$D(\mu)\subseteq H^2$
and the norm on $D(\mu)$ is defined by
\begin{equation}\label{e13}
\|f\|^2_{D(\mu)} = \|f\|_{H^2}^2 + \int_\D |f'(z)|^2 U_\mu(z)\,dA(z)<\infty.
\end{equation}
Dirichlet-type spaces $D(\mu)$ include the classical Hardy space $H^2$ and the Dirichlet space $\mathscr{D}$. When $d\mu(w)=(1-|w|^2)dA(w), w \in \D$, or a positive measure supported on a compact set of $\D$, then $D(\mu)$ equals $H^2$ with equivalent norms (\cite{Al93}). When $d\mu(w)=\frac{|dw|}{2\pi}, w \in \T$, we have $D(\mu) = \mathscr{D}$ (\cite{RS91}).

Dirichlet-type spaces are important model spaces for operators of Dirichlet-type (\cite[P70]{Al93}), in particular the cyclic analytic 2-isometries. Through the work of \cite{RS91} and \cite{Al93} we know that every Dirichlet-type space with a nonnegative superharmonic weight is of the form $D(\mu)$ for some finite nonnegative Borel measure $d\mu$ on $\overline{\D}$. Indeed, there is a one to one correspondence between the set of finite positive Borel measures $d\mu$ on $\overline{\D}$ and the collection of positive superharmonic functions $U_\mu$ on $\D$ (\cite{La72}). For the measures $d\mu$ supported on the unit circle $\T$ these spaces come from \cite{Ri91}. The general case has been considered in \cite{Al93}. The paper \cite{LR15} has an overview of the $D(\mu)$ spaces. Interestingly, the $D(\mu)$ spaces also include the important power weighted Dirichlet spaces $\mathscr{D}_\alpha (0<\alpha<1)$ defined by
\begin{equation*}
\label{e14}
\mathscr{D}_\alpha = \left\{f \in \Hol(\D): \int_\D |f'(z)|^2 (1-|z|^2)^{1-\alpha} dA(z) < \infty\right\}.
\end{equation*}
In fact, if
$$
d\mu(w)= -(1-|w|^2)\partial\overline{\partial} (1-|w|^2)^{1-\alpha} dA(w)\quad w\in {\D}
$$
then by \cite[Proposition IV.1.4]{Al93} one has
$
D(\mu) = \mathscr{D}_\alpha.
$
Even more interestingly, there is a connection between the $D(\mu)$ spaces and the local Dirichlet integral (cf. \cite{RS91, Al93}). More precisely, for $f \in H^2$ and $z \in \overline{\D}$ we define the local Dirichlet integral of $f$ at $z$ by
$$D_z(f) = \int_\T \left|\frac{f(w) - f(z)}{w-z}\right|^2 \frac{|dw|}{2\pi}.$$
If $f$ does not have a nontangential limit at $z \in \T$, then $D_z(f)$ is defined to be $\infty$. By \cite[Theorem IV.1.9]{Al93} we have
\begin{align*}
\int_\D |f'(z)|^2 U_\mu(z)\,dA(z)=\int_{\overline{\D}}D_z(f)\,d\mu(z).
\end{align*}
whence
$$D(\mu) = \left\{f \in H^2: \int_{\overline{\D}}D_z(f)\,d\mu(z) < \infty\right\}.$$
In particular, when $d\mu = \delta_\lambda, \lambda \in \overline{\D}, f \in H^2$, we have
\begin{align}\label{localdir}
D_\lambda(f) = \begin{cases} \int_\D |f'(z)|^2 \bigg(\log\left|\frac{1-\overline{z}\lambda}{\lambda-z}\right|^2\bigg) \frac{1}{1-|\lambda|^2}dA(z) \quad &\lambda \in \D;\\
\int_\D |f'(z)|^2\frac{1-|z|^2}{|1-\overline{z}\lambda|^2}dA(z) \quad &\lambda \in \T.
\end{cases}
\end{align}

\subsection{Corona principle for multiplier algebra $M\big(D(\mu)\big)$}\label{s12}

Given a Banach space $\HB$ of holomorphic functions on $\D$ let
$$M(\HB) = \big\{\phi \in \HB:\  \phi f \in \HB\ \ \forall f\ \in \HB\big\}$$
be the multiplier algebra of $\HB$. It is well-known (cf. \cite{Steg}) that if $H^\infty$ is the space of all bounded holomorphic functions on $\D$ equipped the supremum norm $\|\cdot\|_{H^\infty}$ then
$$
M(\HB)=\begin{cases}
H^\infty\quad&\text{as}\quad \HB=H^2;\\
H^\infty\cap\Big\{\phi:\ \int_{\cup_{k=1}^n S(I_k)}|\phi'|^2\,dA\lesssim \text{Cap}_\frac12\big(\cup_{k=1}^n I_k\big)\Big\} \quad&\text{as}\quad \HB=\mathscr{D},
\end{cases}
$$
where $\text{Cap}_\frac12\big(\cup_{k=1}^n I_k\big)$ is the $\mathcal{L}^2_{\frac12}$-Bessel capacity of the finite union of disjoint subarcs $\{I_k\}_{k=1}^n$ of $\T$ and $$
S(I_k)=\left\{z\in\D:\ \frac{z}{|z|}\in I_k\quad\&\quad 1-\frac{|I_k|}{2\pi}\le|z|<1\right\}
$$
is the Carleson box based on $I_k$ with the arclength $|I_k|$.
Furthermore, a straightforward esimate gives that $\phi\in M\big(D(\mu)\big)$ if and only if  $\phi\in H^\infty$ and $|\phi'|^2U_\mu dA$ is a $D(\mu)$-Carleson measure -i.e.-
$$
\int_{\D}|f(z)|^2|\phi'(z)|^2U_\mu(z)\,dA(z)\lesssim \|f\|^2_{D(\mu)}\quad\forall\quad f\in D(\mu).
$$
Consequently, $M\big(D(\mu)\big)$ is a Banach algebra under the norm
\begin{equation}
\label{e15}
\|\phi\|_{M\big(D(\mu)\big)}=\sup\Big\{\|\phi f\|_{D(\mu)}:\ \ f\in D(\mu)\ \ \&\ \ \|f\|_{D(\mu)}=1\Big\}.
\end{equation}
Thus, it is a natural matter to deal with the algebraic mutiplicative homomorphisms $\Phi\neq 0$ of $M\big(D(\mu)\big)$ onto $\mathbb C$:
$$
\Phi(fg)=\Phi(f)\Phi(g)\ \ \&\ \ \Phi(f+g)=\Phi(f)+\Phi(g)\ \ \forall\ \ f,g\in M\big(D(\mu)\big).
$$
Clearly, each homomorphism of $M\big(D(\mu)\big)$ is a continuous linear functional on $M\big(D(\mu)\big)$ with the norm \eqref{e15} at most one. So, if $\mathcal{M}_{M\big(D(\mu)\big)}$ stands for the class of all such homomorphisms of $M\big(D(\mu)\big)$, then $\mathcal{M}_{M\big(D(\mu)\big)}$ exists as the maximal ideal space of ${M\big(D(\mu)\big)}$. Since each $z\in\mathbb D$ induces a multiplicative linear functional on $M\big(D(\mu)\big)$:
$$
\Phi_z(f)=f(z)\quad\forall\quad f\in M\big(D(\mu)\big),
$$
under the Gelfand topology of $\mathcal{M}_{M\big(D(\mu)\big)}$ (cf. \cite[p.184]{Ga81}) we can readily get that $\D$ is homeomorphically embedded into $\mathcal{M}_{M\big(D(\mu)\big)}$. Moreover, we reveal the following corona principle:
\begin{theorem}
\label{t1} $\D$ is dense in $\mathcal{M}_{M\big(D(\mu)\big)}$ - equivalently - if
$$\{f_1,...,f_n\}\subseteq M\big(D(\mu)\big)\quad\text{and}\quad \inf_{z\in\D}\sum_{j=1}^n|f_j(z)|>0
$$
then
$$
\exists\,\{g_1,...,g_n\}\subseteq M\big(D(\mu)\big)\quad\text{such\ that}\quad \sum_{j=1}^nf_jg_j=1.
$$
\end{theorem}

This result solves an interesting problem within the theory of bounded holomorphic functions and discovers the following facts:

\begin{itemize}
\item [$\rhd$] If $D(\mu)=H^2$, then Theorem \ref{t1} is the Carleson corona theorem in \cite[Theorem 5]{Ca62}. The articles \cite{DKSTW, Ga81, Sa09} are excellent references for the history of the original corona problem and its relatives.

\item [$\rhd$] If $D(\mu)=\mathscr{D}_{0<\alpha\le 1}$, then Theorem \ref{t1} is the Xiao corona theorem in \cite[Theorem 3.4]{Xi98} (see also \cite[P22]{DKSTW}) which was proved by solving a d-bar equation subject to a Carleson measure condition for $\mathscr{D}_{0<\alpha\le 1}$. An equivalent form of this corona theorem can be found in \cite{To91} and has been extended to a higher-dimensional case in \cite{CSW11}.

\item[$\rhd$] When $d\mu$ is a finite sum of atoms $\sum_{j=1}^kc_j\delta_{\zeta_j}$ with $c_j>0\ \&\ \zeta_j\in\T$, Theorem \ref{t1} is reduced to \cite[Theorem 3.1.9]{LarX} which was validated via the following observation
\begin{equation}
\label{e16}
M\Bigg(D\bigg(\sum_{j=1}^kc_j\delta_{\zeta_j}\bigg)\Bigg)=H^\infty\cap D\bigg(\sum_{j=1}^kc_j\delta_{\zeta_j}\bigg).
\end{equation}
In general $M(D(\mu))$ is strictly contained in $H^\infty\cap D(\mu)$.
\end{itemize}

Thanks to the work of Shimorin (\cite{Sh02}), we know that every Dirichlet-type space $D(\mu)$ is a Hilbert space of holomorphic functions whose reproducing kernel is a complete Nevanlinna-Pick kernel (\cite{AM00}). If a Hilbert space has a complete Nevanlinna-Pick kernel, then we call it a complete Nevanlinna-Pick space. Every complete Nevanlinna-Pick space can be embedded in the Drury-Arveson space $H^2_d$ for some cardinal $d$ (\cite[Theorem 8.2]{AM02}). It was shown recently in \cite{AHMR17} that there is a Salas space $\HS$ on $\D$ having a complete Nevanlinna-Pick kernel for which the corona theorem fails for $M(\HS)$. Thus it will be interesting to know for which complete Nevanlinna-Pick spaces the corona theorem holds.
When $\HH$ is a complete Nevanlinna-Pick space, then by the Toeplitz corona theorem (\cite{BTL97, AM02, Sa09}), the corona theorem for its multiplier algebra $M(\HH)$ is equivalent to the baby corona theorem for $\HH$.
The corona theorem for the Drury-Arveson space and the Besov spaces on the ball was proved in \cite{CSW11}. In this paper, we prove the corona theorem for the Dirichlet-type $D(\mu)$ spaces. Since we are dealing with general positive measures $\mu$ on $\overline{\D}$, the standard estimates in studying the corona problems are hard to obtain. In fact, the known cases of the corona theorem are about spaces where the needed estimates can be obtained by exploiting the radial symmetry of the domain. But the $D(\mu)$ spaces in general don't have that symmetry. The strategy in this paper is to use the Cauchy dual of $D(\mu)$ to study the corona theorem for $M(D(\mu))$, and a detailed study of the Cauchy duality is presented in \S \ref{s2a}.
In Section 3 we will establish the corona theorem for $M(D(\mu))$.

The corona theorem for the $\mathscr{D}_{0<\alpha\le 1}$ spaces has been extended to infinitely many functions (\cite{Tr04, KT13}). We will pursue this study for the Dirichlet-type $D(\mu)$ spaces in a future work.

Throughout this paper, for two quantities $A_1$ and $A_2$, we use $A_1 \lesssim A_2$ to mean that there is a constant $C>0$ such that $A_1 \leqslant C A_2$, and write $A_1\approx A_2$ for $A_1 \lesssim A_2 \lesssim A_1$.

\section{Cauchy duality and multiplier algebra}\label{s2a}
\subsection{Cauchy duality}
Suppose $\HH$ is a reproducing kernel Hilbert space of some holomorphic functions on $\D$ and satisfies the five conditions as seen below:
\begin{align}\label{consfcd}
\begin{cases}
1\in\HH;\\
M_z \HH \subseteq \HH&\text{where}~~ M_z f = zf;\\
L_\lambda \HH \subseteq \HH&\text{where}~~ L_\lambda f = \frac{f - f(\lambda)}{z-\lambda};\\
\overline{\D}=\sigma(M_z)&\text{- i.e. - ${\D}$ is dense in the specturm}\ \sigma(M_z)\ \text{of}~~ M_z;\\
\overline{\mathcal P}=\HH&\text{- i.e. - complex polynomial class $\mathcal{P}$ is dense in}~~\HH.
\end{cases}
\end{align}
For each $f \in \HH$ we define the function
$$\phi(\lambda)=(Uf)(\lambda) = \left\langle f, \frac{1}{1-\overline{\lambda}z}\right\rangle_{\HH}$$
with $\langle\cdot,\cdot\rangle_\HH$ being the inner product on $\HH$, and introduce the space
$$\HK = U(\HH)\quad\text{with its norm}\quad
\|\phi\|_\HK = \|f\|_\HH.
$$
Then $\HK$ is a reproducing kernel Hilbert space enjoying (\ref{consfcd}) (cf. \cite[Proposition 5.2]{ARR98}). For any polynomial pair $(p,q) \in \HH\times\HK$ the pairing between $\HH$ and $\HK$ is given by
$$\Big\langle p, q\Big\rangle_{(\HH,\HK)} = \int_\T p(\zeta) \overline{q(\zeta)} \frac{|d\zeta|}{2\pi}.$$
Accordingly, $\HK$ is called the Cauchy dual of $\HH$. Also $\HH = U(\HK)$ is the Cauchy dual of $\HK$. The Cauchy dual of a reflexive Banach space is defined analogously. See \cite[Section 5]{ARR98} for a detailed discussion of the Cauchy duality.

Since $D(\mu)$ is a reproducing kernel Hilbert space of some holomorphic functions satisfying (\ref{consfcd}) (\cite{Al93}), it follows that $D(\mu)$ has a Cauchy dual. Let $$V_\mu(z) = \int_{\overline{\D}} \frac{1-|z|^2}{|1-\overline{\zeta}z|^2} d\mu(\zeta)\quad z\in\D.$$
Let
$$L^2_{a,\mu}=\left\{f \in \Hol(\D): \int_\D |f'(z)|^2 (1-|z|^2)^2 \big(V_\mu(z)\big)^{-1}\,dA(z) < \infty\right\}$$
with norm
$$\|f\|_{L^2_{a,\mu}}^2=|f(0)|^2+\int_\D |f'(z)|^2 (1-|z|^2)^2 \big(V_\mu(z)\big)^{-1}\,dA(z),$$
and let $E(\mu)$ be the closure of the polynomials in $L^2_{a,\mu}$. We will see that $L^2_{a,\mu}$ is a Hilbert space, and so $E(\mu)$ is also a Hilbert space. 
\begin{lemma}\label{intestimate}
Let $s>-1, r, t \geq0$ and $r+t-s>2$. If $t < s+2 < r$, then
$$\int_\D \frac{(1-|z|^2)^s}{|1-\lambda\overline{z}|^r |1-\overline{\zeta}z|^t} dA(z) \lesssim \frac{1}{(1-|\lambda|^2)^{r-s-2}|1-\overline{\zeta}\lambda|^t}\ \ \forall\ \ (\lambda,\zeta)\in\mathbb D\times \overline{\mathbb D}.
$$
\end{lemma}
\begin{proof} When $\zeta \in \D$, it is proved in \cite{OF96, Zh08}. When $\zeta \in \T$, we apply Fatou's lemma to obtain the conclusion.
\end{proof}

\begin{lemma}\label{propeose} We have:
\begin{itemize}
\item [\rm (i)] $E(\mu)$ contains $H^2$ and is contained in the weighted Bergman space
$$
L^2_a\big((1-|z|^2)\,dA(z)\big) = \left\{f \in \Hol(\D): \int_\D |f(z)|^2 (1-|z|^2)\,dA(z) < \infty\right\};
$$
\item [\rm (ii)] $E(\mu)$ is a reproducing kernel Hilbert space;
\item[\rm (iii)] $E(\mu)$ is the Cauchy dual of $D(\mu)$;
\item[\rm (iv)] The multiplier algebra of $E(\mu)$ is $H^\infty$ - i.e. - $M\big(E(\mu)\big) = H^\infty.$
\end{itemize}
\end{lemma}
\begin{proof}

(i) Note that
$$\frac{1-|z|^2}{4}\mu(\overline{\D}) \leq V_\mu(z) \leq \frac{2\mu(\overline{\D})}{1-|z|^2},$$
by using Fatou's lemma we conclude that $L^2_{a,\mu}$ and $E(\mu)$ are Hilbert spaces and
$$
H^2 \subseteq E(\mu) \subseteq L^2_{a,\mu}\subseteq L^2_a\big((1-|z|^2)\,dA(z)\big).
$$

(ii) This follows from (i).

(iii) First we look at the following observation. If $f \in H^2$, then
\begin{align*}
\int_\D |f'(\zeta)|^2 \frac{(1-|z|^2)(1-|\zeta|^2)}{|1-\overline{\zeta}z|^2} dA(\zeta) &\leq \int_\D |f'(\zeta)|^2 \bigg(\log\left|\frac{1-\overline{\zeta}z}{z-\zeta}\right|^2\bigg) dA(\zeta)\\
& \lesssim \int_\D |f'(\zeta)|^2 \frac{(1-|z|^2)(1-|\zeta|^2)}{|1-\overline{\zeta}z|^2} dA(\zeta),
\end{align*}
and hence
\begin{align*}
\int_\D D_\zeta(f)\,d\mu(\zeta) &= \int_\D |f'(\zeta)|^2 \int_\D \bigg(\log\left|\frac{1-\overline{\zeta}z}{z-\zeta}\right|^2\bigg) \frac{d\mu(z)}{1-|z|^2}dA(\zeta)\\
&\approx \int_\D |f'(\zeta)|^2 \int_\D \frac{1-|\zeta|^2}{|1-\overline{\zeta}z|^2} d\mu(z)\,dA(\zeta).
\end{align*}
It follows that $f \in D(\mu)$ if and only if
\begin{align*}
&\int_\D |f'(\zeta)|^2 \int_\D \frac{1-|\zeta|^2}{|1-\overline{\zeta}z|^2} d\mu(z)\,dA(\zeta) + \int_\D |f'(\zeta)|^2 \int_\T \frac{1-|\zeta|^2}{|1-\overline{\zeta}z|^2} d\mu(z)\,dA(\zeta)\\
&\quad= \int_\D |f'(\zeta)|^2 \int_{\overline{\D}} \frac{1-|\zeta|^2}{|1-\overline{\zeta}z|^2} d\mu(z)\,dA(\zeta) < \infty.
\end{align*}
Thus
$$D(\mu) = \left\{f \in \Hol(\D): \int_{{\D}} |f'(z)|^2 V_\mu(z)\,dA(z) < \infty\right\},$$
and an equivalent norm on $D(\mu)$ can be defined by
\begin{equation}\label{e14}
\|f\|^2_{D(\mu)} = \|f\|_{H^2}^2 + \int_\D |f'(z)|^2 V_\mu(z)\,dA(z).
\end{equation}
In the following we will use this equivalent norm (\ref{e14}) on $D(\mu)$.

By the definition of the Cauchy duality, we have that the Cauchy dual of $D(\mu)$ is
\begin{align*}
\HK & = \Bigg\{\phi:\ \phi(\lambda) = \Big\langle f, \frac{1}{1-\overline{\lambda}z}\Big\rangle_{D(\mu)}, f \in D(\mu) \Bigg\}\\
& = \Bigg\{\phi:\ \phi(\lambda) = f(\lambda) + \lambda\int_\D f'(z) {(1-\lambda \overline{z})^{-2}}V_\mu(z) dA(z), f \in D(\mu) \Bigg\}
\end{align*}
with norm $\|\phi\|_{\HK} = \|f\|_{D(\mu)}$.
Now we show $\HK = E(\mu)$.

Note that for $f\in H^2$ one has
\begin{align*}
\|f\|_{H^2}^2 &= |f(0)|^2 + 2\int_\D |f'(z)|^2 \log \frac{1}{|z|} dA(z)\\
& \approx |f(0)|^2 + \int_\D |f'(z)|^2 (1-|z|^2) dA(z).
\end{align*}
It is then clear that $E(\mu) \subseteq \HK$ under the Cauchy duality. For the reverse inclusion, if $\phi \in \HK$, then there is $f \in D(\mu)$ such that
$$\phi(\lambda) = f(\lambda) + \int_\D f'(z) \frac{\lambda}{(1-\lambda \overline{z})^2}V_\mu(z) dA(z),
$$
and hence
$$\phi'(\lambda) = f'(\lambda) + \int_\D f'(z) \frac{1+\lambda \overline{z}}{(1-\lambda \overline{z})^3}V_\mu(z) dA(z).$$
Since $f \in D(\mu)$, the above-proven (i) yields $f \in D(\mu) \subseteq H^2 \subseteq E(\mu).$
Suppose
$$g(\lambda) = \int_\D  \frac{|f'(z)|}{|1-\lambda \overline{z}|^3}V_\mu(z) dA(z).
$$
Then, it is enough to show
$$\int_\D |g(\lambda)|^2 (1-|\lambda|^2)^2 \big(V_\mu(\lambda)\big)^{-1}dA(\lambda) < \infty.$$
We will use Schur's Theorem (cf. \cite{Zhu07}) to achieve this. Let $h(z) = (1-|z|^2)^{-1/2}$.
By Lemma \ref{intestimate}, we have
\begin{align*}
\int_\D  \frac{h(z)}{|1-\lambda \overline{z}|^3}V_\mu(z) dA(z) & = \int_{\overline{\D}} \int_\D  \frac{(1-|z|^2)^{1/2}}{|1-\lambda \overline{z}|^3|1-\overline{\zeta}z|^2} dA(z)d\mu(\zeta)\\
&\lesssim \int_{\overline{\D}} \frac{1}{(1-|\lambda|^2)^{1/2}|1-\overline{\zeta}\lambda|^2}d\mu(\zeta)\\
& = (1-|\lambda|^2)^{-3/2} V_\mu(\lambda).
\end{align*}
Also
\begin{align*}
&\int_\D  \frac{1}{|1-\lambda \overline{z}|^3}(1-|\lambda|^2)^{-3/2} V_\mu(\lambda) (1-|\lambda|^2)^2 \big(V_\mu(\lambda)\big)^{-1} dA(\lambda)\\
& = \int_\D  \frac{(1-|\lambda|^2)^{1/2}}{|1-\lambda \overline{z}|^3} dA(\lambda)\\
& \lesssim (1-|z|^2)^{-1/2} = h(z).
\end{align*}
Thus
\begin{align*}
|g(\lambda)|^2& = \left|\int_\D  \frac{|f'(z)|}{|1-\lambda \overline{z}|^3}V_\mu(z) dA(z)\right|^2\\
& \lesssim \int_\D  \frac{|f'(z)|^2h(z)^{-1}}{|1-\lambda \overline{z}|^3}V_\mu(z) dA(z) \int_\D  \frac{h(z)}{|1-\lambda \overline{z}|^3}V_\mu(z) dA(z)\\
& \lesssim (1-|\lambda|^2)^{-3/2} V_\mu(\lambda) \int_\D  \frac{|f'(z)|^2h(z)^{-1}}{|1-\lambda \overline{z}|^3}V_\mu(z) dA(z),
\end{align*}
and Fubini's theorem derives
\begin{align*}
&\int_\D |g(\lambda)|^2 (1-|\lambda|^2)^2 \big(V_\mu(\lambda)\big)^{-1}dA(\lambda)\\
& \lesssim \int_\D |f'(z)|^2 V_\mu(z) dA(z)\\
& \leqslant \|f\|_{D(\mu)}^2 < \infty.
\end{align*}
This establishes the desired conclusion.

(iv) Suppose that $L$ is the backward shift operator:
$$L(f)(z) = \frac{f(z)-f(0)}{z}\quad\forall\quad f \in D(\mu).
$$
Then $M_z^*\big|_{E(\mu)} = L\big|_{D(\mu)}$.

First, let us show $
\|M_z\| \leq 1$ on $E(\mu).$
In fact, we are required to verify
$
\|L\| \leq 1$ on $D(\mu).$
According to \cite{RS91} and
$$
g(z) = g(0) + z(Lg)(z)\quad\forall\quad g \in D(\mu),
$$
we have that
$$D_\lambda(g) = |Lg(\lambda)|^2 + D_\lambda (Lg)\ \ \text{holds}~~~\mu-\text{a.e.}~~ \lambda \in \T.$$
Meanwhile, if $\lambda \in \D$, then
\begin{align*}
D_\lambda(g) &= \int_\T \left|\frac{zLg(z) - \lambda Lg(\lambda)}{z-\lambda}\right|^2 \frac{|dz|}{2\pi}\\
& = \int_\T \frac{|Lg(z)|^2 - |\lambda Lg(\lambda)|^2}{|1-\overline{\lambda}z|^2}\frac{|dz|}{2\pi}\\
& = \int_\T \frac{|Lg(z)|^2 - |Lg(\lambda)|^2}{|1-\overline{\lambda}z|^2}\frac{|dz|}{2\pi} + \int_\T \frac{|Lg(\lambda)|^2(1-|\lambda|^2)}{|1-\overline{\lambda}z|^2}\frac{|dz|}{2\pi}\\
& = D_\lambda (Lg) + |Lg(\lambda)|^2.
\end{align*}
Since $\|Lg\|_{H^2} \leq \|g\|_{H^2},
$
it follows that
$$\|L\| \leq 1\quad\text{on}\quad D(\mu)\quad\&\quad \|M_z\| \leq 1\quad\text{on}\quad E(\mu).
$$

Second, we are about to show
$M\big(E(\mu)\big) = H^\infty.$
Suppose $\HH$ is a reproducing kernel Hilbert space with kernel $K^\HH$. From \cite{AM02} or \cite{PR16}, we have that a function $\phi$ is a multiplier of $\HH$ with norm at most one if and only if
$$\big(1-\overline{\phi(w)}\phi(z)\big)K_w^{\HH}(z)$$
is positive definite.
Upon letting $K_w^{E(\mu)}(z)$ be the reproducing kernel of $E(\mu)$, we obtain that $$(1-\overline{w}z)K_w^{E(\mu)}(z)$$ is positive definite. Note that $\big({1-\overline{w}z}\big)^{-1}$ is the reproducing kernel of $H^2$. So, if $\phi \in M(H^2)$, then
$$\big(\|\phi\|_{M(H^2)}^2-{\overline{\phi(w)}\phi(z)\big)}\big({1-\overline{w}z}\big)^{-1}$$
is positive definite.
The Schur product theorem then implies that
$$\big(\|\phi\|_{M(H^2)}^2-\overline{\phi(w)}\phi(z)\big) K_w^{E(\mu)}(z)$$
is positive definite.
Thus $H^\infty = M(H^2) \subseteq M\big(E(\mu)\big).$
Conversely, let $\phi \in M\big(E(\mu)\big)$. Since $M_\phi^* K_w^{E(\mu)} = \overline{\phi(w)}K_w^{E(\mu)},$
we have $M\big(E(\mu)\big) \subseteq H^\infty=M(H^2),$
thereby reaching the identification.
\end{proof}
\begin{remark}\label{vonneumann}
Once we know $\|M_z\| \leq 1$ on $E(\mu)$, we can also apply von Neumann's inequality to conclude that $M\big(E(\mu)\big) = H^\infty$.
\end{remark}

\subsection{Multiplier algebra}
For a finite positive Borel measure $d\mu$ on $\overline{\D}$ let
$$HD(\mu) = \left\{f \in L^2(\T): \int_{\overline{\D}} D_\lambda\big(P(f)\big)\,d\mu(\lambda) < \infty\right\},$$
where $P(f)$ is the harmonic extension of $f$. $HD(\mu)$ is called the harmonic Dirichlet-type space. The norm on $HD(\mu)$ is defined by
$$\|f\|^2_{HD(\mu)}= \|f\|^2_{L^2(\T)} + \int_{\overline{\D}} D_\lambda\big(P(f)\big)\,d\mu(\lambda).$$
According to \cite[Lemma 4.1]{RS14} we have
$$D_\lambda(f + \overline{g}) = D_\lambda(f) + D_\lambda(\overline{g})\quad\forall\quad (f,g,\lambda)\in H^2\times H^2\times\overline{\D}
$$
whence
$$
\begin{cases}D(\mu) \perp \overline{D(\mu)_0};\\
HD(\mu) = D(\mu) + \overline{D(\mu)_0};\\
D(\mu)_0 = \big\{f \in D(\mu): f(0) = 0\big\}.
\end{cases}
$$
 Recall that
$$
\int_\D |f'(z)|^2 U_\mu (z)\,dA(z) = \int_{\overline{\D}} D_\lambda(f)\,d\mu(\lambda)\quad\forall\quad f\in H^2.
$$
So, if
$$\nabla f (z)= \left(\frac{\partial f}{\partial x}, \frac{\partial f}{\partial y}\right)\quad\forall\quad z=x+iy,
$$
then
$$
\begin{cases}HD(\mu) = \left\{f \in L^2(\T): \int_\D |\nabla P(f)(z)|^2 U_\mu (z)\,dA(z) < \infty\right\};\\
\|f\|_{HD(\mu)}^2 = \|f\|_{L^2(\T)}^2 + \frac{1}{2}\int_\D |\nabla P(f)(z)|^2 U_\mu (z)\,dA(z).
\end{cases}
$$
Let
$$
M\big(HD(\mu)\big) = \bigg\{\phi \in HD(\mu): f\in HD(\mu)\Rightarrow\phi f \in HD(\mu)\bigg\}
$$
be the multiplier algebra of $HD(\mu)$. Note that when $f \in HD(\mu), \phi \in M\big(HD(\mu)\big)$ in general we have $P(\phi f) (z)\neq P(\phi) (z) P(f)(z), z \in \D$.
We say that a positive Borel measure $\nu$ is a $D(\mu)$-Carleson measure if
$$\int_{\D}|f(z)|^2 d\nu(z)\lesssim \|f\|^2_{D(\mu)}\quad\forall\quad f\in D(\mu).$$

Recall that $$V_\mu(z) = \int_{\overline{\D}} \frac{1-|z|^2}{|1-\overline{\zeta}z|^2} d\mu(\zeta)\quad z\in\D,$$
and, (\ref{e13}) and (\ref{e14}) are equivalent norms on $D(\mu)$.

\begin{prop}\label{mulohimdm1} We have
\begin{enumerate}
\item $\phi\in M\big(D(\mu)\big)$ if and only if
	$\phi\in H^\infty$ and $|\phi'|^2U_\mu\,dA$ (or $|\phi'|^2V_\mu\,dA$) is a $D(\mu)$-Carleson measure.
\item $M\big(D(\mu)\big)\subseteq M\big(HD(\mu)\big)$.
\end{enumerate}
\end{prop}
\begin{proof}
(i) This is not hard to verify by the definition.

(ii) Let $\phi\in M\big(D(\mu)\big)$.
	 If $f\in HD(\mu)$, then there are $f_1, f_2 \in D(\mu)$ such that
	$$
	f_2(0) = 0\quad\&\quad f = f_1 +  \overline{f_2}.
	$$
	Since $\phi f_1 \in D(\mu)\subseteq HD(\mu)$,
	it is enough to show
	$
	\phi \overline{f_2} \in HD(\mu).
	$
	Recall that $E(\mu)$ is the Cauchy dual of $D(\mu)$ and the polynomials are dense in $E(\mu)$. Let $g = g_1 + \overline{g_2}$, where $g_1, g_2 \in E(\mu)$ are polynomials, and $g_2(0) = 0$.
	By Green's theorem
\begin{align}\label{green}
\int_\T u(z) \frac{|dz|}{2\pi} = u(0) + 2\int_\D \partial\overline{\partial} u(z) \bigg(\log\frac{1}{|z|}\bigg)\, dA(z)\quad\forall\quad u \in C^2(\overline{\D})
\end{align}
we have (the functions in the following are not $C^2$, but we can use a limiting argument to see that the following holds)
	\begin{align}\label{estimate}
	&\left|\int_\T \phi(z) \overline{f_2(z) g(z)} \frac{|dz|}{2\pi}\right|\notag\\
	&\ \leq \left|\int_\T \phi(z) \overline{f_2(z) g_1(z)} \frac{|dz|}{2\pi} \right| + \left|\int_\T \phi(z) \overline{f_2(z)} g_2(z) \frac{|dz|}{2\pi} \right|\\
	&\ = 2 \left|\int_\D \partial\overline{\partial} (\phi \overline{f_2 g_1})(z) \bigg(\log\frac{1}{|z|}\bigg) dA(z) \right| + 2\left|\int_\D \partial\overline{\partial} (\phi \overline{f_2} g_2)(z) \bigg(\log\frac{1}{|z|}\bigg) dA(z) \right|\notag\\
	&\ \lesssim\left|\int_\D \frac{\phi'(z)\big(\overline{f_2' g_1} + \overline{f_2 g_1'} \big)(z)} {\bigg(\log\frac{1}{|z|}\bigg)^{-1}}dA(z)\right|+\left|\int_\D \frac{\big(\phi' g_2 + \phi g_2'\big)(z) \overline{f_2'}(z)}{\bigg(\log\frac{1}{|z|}\bigg)^{-1}}dA(z)\right|.\notag
	\end{align}
Upon utilizing part (i) we obtain
\begin{align}\label{mulfdmest}
\int_\D |\phi'(z) f_2(z)|^2 V_\mu(z)\,dA(z) \lesssim \|M_{\phi}\|^2 \|f_2\|_{D(\mu)}^2\lesssim\|f_2\|_{D(\mu)}^2.
\end{align}
An application of Lemma \ref{propeose}(iv), together with the product rule $\phi' g = (\phi g)' - \phi g'$ yields
\begin{equation}\label{mulfebest}
\begin{cases}
\phi \in M\big(E(\mu)\big);\\
\int_\D |\phi'(z) g_i(z)|^2 (1-|z|^2)^2 \big(V_\mu(z)\big)^{-1}\,dA(z) \lesssim \|g_i\|_{E(\mu)}^2, i = 1, 2.
\end{cases}
\end{equation}
Meanwhile, it is not hard to check
$$\int_\D |\phi'(z) g_i(z)|^2 \bigg(\log\frac{1}{|z|}\bigg)^2 \frac{dA(z)}{V_\mu(z)} \lesssim \int_\D |\phi'(z) g_i(z)|^2 (1-|z|^2)^2 \frac{dA(z)}{V_\mu(z)}, i =1, 2.
$$
Thus applying Cauchy-Schwarz inequality we conclude from (\ref{estimate}) that
$$\left|\int_\T \phi(z) \overline{f_2(z)g(z)} \frac{|dz|}{2\pi}\right|  \lesssim \|f_2\|_{D(\mu)} \big(\|g_1\|_{E(\mu)} + \|g_2\|_{E(\mu)}\big).$$
Consequently, we achieve
	$$
	\phi \overline{f_2} \in HD(\mu)\quad\&\quad \|\phi \overline{f_2}\|_{HD(\mu)} \lesssim \|f_2\|_{D(\mu)}.
	$$
\end{proof}

\section{Corona principle}
Suppose $$f = (f_1, \ldots, f_n) \in M\big(D(\mu)\big)\times\cdots\times M\big(D(\mu)\big)=\Big(M\big(D(\mu)\big)\Big)^n,
$$
we denote the operator $$M_f: D(\mu) \times \cdots \times D(\mu) \rightarrow D(\mu)$$ by
$$
g=(g_1,...,g_n)\mapsto M_f g = \sum_{i=1}^n f_i g_i.
$$
Now Theorem \ref{t1} follows from the following corona principle.

\begin{theorem}\label{corthmfducd}
Let $$f=(f_1, \ldots, f_n) \in \Big(M\big(D(\mu)\big)\Big)^n.
$$
Then the following are equivalent:
\begin{itemize}
\item[\rm(i)] $M_f: M\big(D(\mu)\big) \times \cdots \times M\big(D(\mu)\big) \rightarrow M(D(\mu))$ is onto;
\item[\rm(ii)] $M_f: D(\mu) \times \cdots \times D(\mu) \rightarrow D(\mu)$ is onto;
\item[\rm(iii)] There exists a $\delta > 0$ such that
\begin{align}\label{corcondu}
|f(z)|^2 = \sum_{i=1}^n |f_i(z)|^2 \geq \delta^2 > 0\quad \forall\quad z \in \D.
\end{align}
\end{itemize}
\end{theorem}
\begin{proof} It is shown in \cite{Sh02} that $D(\mu)$ is a reproducing kernel Hilbert space with a complete Nevanlinna-Pick kernel. Thus by the Toeplitz corona theorem (\cite{BTL97, AM02, Sa09}), we have that (i) and (ii) are equivalent.

That (ii) implies (iii) is standard. As a matter of fact, let $\delta > 0$. If (ii) is valid, then for every $h \in D(\mu)$ there are $\{g_1,..., g_n\} \subseteq D(\mu)$
such that
$$
\begin{cases}
\sum_{j=1}^n f_jg_j= h;\\
\sum_{j=1}^n \|g_j\|_{D(\mu)}^2 \leq {\delta^{-2}}\|h\|_{D(\mu)}^2,
\end{cases}
$$
and hence
$$
M_fM_f^* - \delta^2 I\geq 0,
$$
where $I$ is the identity map. If $K_z$ is the reproducing kernel of $D(\mu)$, then
\begin{align*}
\delta^2\Big\langle K_z, K_z\Big\rangle_{D(\mu)} \leq \Big\langle M_fM_f^* K_z, K_z\Big\rangle_{D(\mu)} = |f(z)|^2 \Big\langle K_z, K_z\Big\rangle_{D(\mu)},
\end{align*}
and hence
$$
\left(\sum_{j=1}^n|f_j(z)|^2\right)^\frac12=|f(z)| \geq \delta,
$$
namely, (iii) holds.

Conversely, we show that (iii) implies (ii). Suppose that (iii) holds. Then the family of finitely many functions $\{f_1,..., f_n\}\subseteq M\big(D(\mu)\big)$
obeys (\ref{corcondu}). (ii) will follow from validating that if $h \in D(\mu)$ then there is a family of finitely many functions
$\{g_1,..., g_n\} \subseteq D(\mu)$ such that $\sum_{j=1}^n f_jg_j= h.$
Upon utilizing the normal family argument, we may assume that $f_1,..., f_n, h$ are holomorphic in $\overline{\D}$. Let
$$\phi_j(z) = \frac{\overline{f_j(z)}}{|f(z)|^2} ~~ \forall ~~ j\in\{1,...,n\}.$$
Then $\phi_j h$ are non-holomorphic solution of the equation $\sum_{j=1}^n f_j \phi_j h = h.$
To make them holomorphic, we define
\begin{align}\label{corsoldcah}
g_j(z) = \phi_j(z)h(z) + \sum_{k=1}^n \big(a_{jk}(z) - a_{kj}(z)\big) f_k(z),
\end{align}
where
$$a_{jk}(z) = \int_\D \frac{(\phi_j\overline{\partial}\phi_kh)(w)}{z-w} dA(w)$$
is a restatement of the fact that, considered as the Cauchy transform of $\phi_j\overline{\partial}\phi_kh$, $w^{-1}$ is a fundamental solution of the Cauchy-Riemann operator $\overline{\partial}$.
Since $\overline{\partial} a_{jk} = \phi_j\overline{\partial}\phi_kh,$
we have that $g_1,...,g_n$ are holomorphic and $\sum_{j=1}^n f_j(z) g_j(z)$ $= h(z).$
Suppose
$$\|M_{f_j}\| \leq 1\quad\forall\quad j = 1,..., n.
$$
Next we show that there is a constant $C(n,\delta)>0$ such that
$$g_j \in D(\mu)\quad\&\quad \|g_j\|_{D(\mu)} \leq C(n,\delta).
$$

First we prove
$$a_{jk}|_\T \in \overline{D(\mu)}.
$$
Let $b_{jk}(z) = P\big(a_{jk}|_\T\big)$
be the harmonic extension of $a_{jk}|_\T$. Then $b_{jk}$ is anti-holomorphic with
$b_{jk}(0)=0$. Lemma \ref{propeose}(iii) ensures that $E(\mu)$ is the Cauchy dual of $D(\mu)$. Let $p\in E(\mu)$ be a polynomial with $p(0) = 0$. Since polynomials are dense in $E(\mu)$, it is enough to show
$$
\left|\int_\T b_{jk}(z) p(z) \frac{|dz|}{2\pi}\right| \lesssim\|p\|_{E(\mu)}.
$$
By Green's theorem (\ref{green})
we have
\begin{align}\label{estfaijdmu}
&\int_\T b_{jk}(z) p(z) \frac{|dz|}{2\pi}\nonumber\\
&\quad = 2\int_\D \partial\overline{\partial} (a_{jk}p)(z) \bigg(\log\frac{1}{|z|}\bigg) dA(z)\notag\\
&\quad= 2 \int_\D \partial(\phi_j\overline{\partial}\phi_k hp) (z)\bigg(\log\frac{1}{|z|}\bigg) dA(z)\\
&\quad= 2 \int_\D \Big(\partial(\phi_j\overline{\partial}\phi_k) h(z)p(z) + \phi_j\overline{\partial}\phi_k\big(h'(z) p(z) + h(z) p'(z)\big)\Big) \bigg(\log\frac{1}{|z|}\bigg)\,dA(z).\notag
\end{align}
Note that
$$\phi_j \overline{\partial}\phi_k = \frac{\overline{f_j}}{|f|^2}\left(\frac{\overline{f_k'}}{|f|^2} - \frac{\overline{f_k}\sum_{l=1}^n f_l \overline{f_l'}}{|f|^4}\right)
$$
and
\begin{align*}
\partial(\phi_j\overline{\partial}\phi_k)
&= \frac{-\overline{f_j}\sum_{l=1}^n f_l' \overline{f_l}}{|f|^4}\left(\frac{\overline{f_k'}}{|f|^2} - \frac{\overline{f_k}\sum_{l=1}^n f_l \overline{f_l'}}{|f|^4}\right)\\
&\ \ + \frac{\overline{f_j}}{|f|^2}\left( \frac{-\overline{f_k'}\sum_{l=1}^n f_l' \overline{f_l}}{|f|^4} + \frac{2\overline{f_k}\sum_{l=1}^n f_l \overline{f_l'} \sum_{l=1}^n f_l' \overline{f_l}}{|f|^6} -\frac{\overline{f_k}\sum_{l=1}^n f_l' \overline{f_l'}}{|f|^4}\right).
\end{align*}
So, it follows from (\ref{estfaijdmu}) that
\begin{align*}
&\left|\int_\T b_{jk}(z) p(z) \frac{|dz|}{4\pi}\right|\notag\\
&\ \ \leq 2\int_\D \Bigg(\frac{\sum_{l=1}^n |f_l'|}{|f|^2}\Bigg)\left(\frac{|f_k'|+ \sum_{l=1}^n |f_l'|}{|f|^2}\right)|h(z)p(z)|\left(\log\frac{1}{|z|}\right)dA(z)\\
&\ \ \ +2\int_{{\D}}\left( \frac{|f_k'|\sum_{l=1}^n|f_l'|+2(\sum_{l=1}^n|f_l'|)^2+\sum_{l=1}^n|f_l'|^2}{|f|^4}\right)|h(z)p(z)|\left(\log\frac1{|z|}\right)dA(z)\notag\\
&\ \ \ + 2\int_{{\D}}\Bigg(\frac{|f_k'|+\sum_{l=1}^n|f_l'|}{|f|^3}\Bigg)\big(|h'(z)p(z)| + |h(z)p'(z)|\big)\bigg(\log\frac{1}{|z|}\bigg) dA(z)\notag\\
&\lesssim (n^2 \delta^{-4} + n \delta^{-3}) \|h\|_{D(\mu)}\|p\|_{E(\mu)}\\
& \lesssim n^2 \delta^{-4}\|h\|_{D(\mu)}\|p\|_{E(\mu)},
\end{align*}
where in the second to the last line we used (\ref{mulfdmest}) and (\ref{mulfebest}).
Consequently,
$$
\|b_{jk}\|_{\overline{D(\mu)}} \lesssim n^2\delta^{-4} \|h\|_{D(\mu)}.
$$

Now we demonstrate that
$$g_j\in D(\mu)\quad\&\quad \|g_j\|_{D(\mu)} \lesssim C(n,\delta)\|h\|_{D(\mu)}
$$
holds for some positive constant $C(n,\delta)$ depending only on $n$ and $\delta$. Evidently, it is enough to show
$$\left|\int_\T g_j(z) \overline{\kappa(z)} \frac{|dz|}{2\pi}\right| \lesssim C(n,\delta) \|\kappa\|_{E(\mu)}\quad\forall\quad\text{polynomial }\kappa\in E(\mu).
$$
Observe that
\begin{align*}
&\left|\int_\T g_j(z) \overline{\kappa(z)} \frac{|dz|}{2\pi}\right|\\
&\ \ = \left|\int_\T \left(\phi_j(z) h(z)\overline{\kappa(z)} + \sum_{k=1}^n (a_{jk}(z) - a_{kj}(z)) f_k(z) \overline{\kappa(z)}\right)\frac{|dz|}{2\pi}\right|\\
&\ \  \leq \left|\int_\T \phi_j(z) h(z)\overline{\kappa(z)}\frac{|dz|}{2\pi}\right| + \left|\int_\T \sum_{k=1}^n (a_{jk}(z) - a_{kj}(z)) f_k(z) \overline{\kappa(z)}\frac{|dz|}{2\pi}\right|.
\end{align*}
So we are required to estimate the above two terms. For the first term we use Green's theorem to calculate
\begin{align*}
&\int_\T \phi_j(z) h(z)\overline{\kappa(z)}\frac{|dz|}{4\pi}-\phi_j(0) h(0)\overline{\kappa(0)}\\
&\ \ = 2\int_\D \Big(\partial\overline{\partial}(\phi_j h\overline{\kappa})(z)\Big) \bigg(\log\frac{1}{|z|}\bigg) dA(z)\\
&\ \ = 2\int_\D \Big(\big(\partial\overline{\partial}\phi_j(z)\big)h(z)\overline{\kappa}(z)+\big(\overline{\partial}\phi_j\big(z)\big) h'(z)\overline{\kappa}(z)\Big)\bigg(\log\frac{1}{|z|}\bigg) dA(z)\\
&\quad + 2\int_\D	
	\Big(\big(\partial\phi_j(z)\big) h(z)\overline{\kappa'(z)}+ \phi_j(z) h'(z) \overline{\kappa'(z)}\Big)\bigg(\log\frac{1}{|z|}\bigg) dA(z).
\end{align*}
Note that
$$
\begin{cases}
\phi_j = \frac{\overline{f_j}}{|f|^2},\quad
\overline{\partial}\phi_j = \left(\frac{\overline{f_j'}}{|f|^2} - \frac{\overline{f_j}\sum_{l=1}^n f_l \overline{f_l'}}{|f|^4}\right),\quad
\partial \phi_j = -\frac{\overline{f_j}\sum_{l=1}^n f_l'\overline{f_l}}{|f|^4},\\
\partial \overline{\partial}\phi_j = -\frac{\overline{f_j'}\sum_{l=1}^n f_l'\overline{f_l}}{|f|^4} + \frac{\overline{f_j}\sum_{l=1}^n f_l \overline{f_l'} \sum_{l=1}^n f_l'\overline{f_l}}{|f|^6} - \frac{\overline{f_j}\sum_{l=1}^n f_l' \overline{f_l'}}{|f|^4}.
\end{cases}
$$
Thus, using (\ref{mulfdmest}) and (\ref{mulfebest}) we obtain
\begin{align*}
&\left|\int_\T \phi_j(z) h(z)\overline{k(z)}\frac{|dz|}{2\pi}\right|\\
& \le\delta^{-1}\|h\|_{D(\mu)}\|\kappa\|_{E(\mu)} + 2 \int_\D \Bigg(\bigg({|f_j'| \sum_{l=1}^n|f_l'| + \big(\sum_{l=1}^n|f_l'|\big)^2 + \sum_{l=1}^n|f_l'|^2}\bigg)\Bigg(\frac{|h\kappa|}{|f|^3}\Bigg)\\
&\hspace{0.5cm} +  \Bigg(\frac{|f_j'| + \sum_{l=1}^n|f_l'|}{|f|^2}\Bigg)|h'\kappa| + \Bigg({\sum_{l=1}^n|f_l'|}\Bigg)\Bigg(\frac{|h\kappa'|}{|f|^2}\Bigg) + \frac{|h'\kappa'|}{|f|}\Bigg) \bigg(\log\frac{1}{|z|}\bigg)\, dA(z)\\
& \lesssim \left({n^2}{\delta^{-3}} + {n}{\delta^{-2}}+\delta^{-1}\right) \|h\|_{D(\mu)}\|\kappa\|_{E(\mu)}\\
& \lesssim n^2 \delta^{-3} \|h\|_{D(\mu)}\|\kappa\|_{E(\mu)}.
\end{align*}
Upon recalling $b_{jk}(z) = P\big(a_{jk}|_\T\big), \|b_{jk}\|_{\overline{D(\mu)}} \lesssim n^2 \delta^{-4}\|h\|_{D(\mu)},$
noticing
\begin{align*}
&\int_\T  b_{jk}(z) f_k(z) \overline{\kappa(z)}\frac{|dz|}{2\pi}\\
&\ \ =  2\int_\D \partial\overline{\partial}\big(b_{jk} f_k \overline{\kappa}\big) (z) \bigg(\log \frac{1}{|z|}\bigg)\,dA(z)\\
&\ \  =  2 \int_\D \Big(\overline{\partial}b_{jk}(z) \overline{\kappa}(z) + b_{jk}(z) \overline{\kappa'(z)}\Big) f_k'(z) \bigg(\log \frac{1}{|z|}\bigg)\,dA(z),
\end{align*}
and applying both (\ref{mulfdmest}) and (\ref{mulfebest}), for the second term we obtain
\begin{align*}
&\left|\int_\T \sum_{k=1}^n a_{jk}(z) f_k(z) \overline{\kappa(z)}\frac{|dz|}{2\pi}\right|\\
& \leq 2\sum_{k=1}^n \left|\int_\D\Big(\overline{\partial}b_{jk}(z) \overline{\kappa(z)} + b_{jk}(z) \overline{\kappa'(z)}\Big) f_k'(z) \bigg(\log \frac{1}{|z|}\bigg)\,dA(z)\right|\\
& \lesssim \sum_{k=1}^n \|b_{jk}\|_{\overline{D(\mu)}} \|\kappa\|_{E(\mu)}\\
& \lesssim n^3 \delta^{-4}\|h\|_{D(\mu)} \|\kappa\|_{E(\mu)},
\end{align*}
thereby reaching
$$g_j \in D(\mu)\quad\text{with}\quad
\|g_j\|_{D(\mu)} \lesssim n^3\delta^{-4}\|h\|_{D(\mu)},
$$
as desired.
\end{proof}
In proving Theorem \ref{corthmfducd}, we obtain
$$
g_j= \phi_jh+ \sum_{k=1}^n (a_{jk} - a_{kj}) f_k\in D(\mu).
$$
Here is a slightly different argument. Since we have shown
$a_{jk}|_\T \in \overline{D(\mu)},
$
Proposition \ref{mulohimdm1} implies
$
a_{jk}|_\T  f_k \in HD(\mu)$.
Note that
$\phi_j = {\overline{f_j}}{|f|^{-2}}
$
and similarly
$
\overline{f_j} h \in HD(\mu).
$
So, by using the Cauchy dual space $E(\mu)$ one can show that
$$
{\kappa}{|f|^{-2}} \in HD(\mu)\ \ \forall\ \ \kappa \in HD(\mu)\ \ \text{
and so}\ \ g_j \in D(\mu).
$$

\

\noindent \textbf{Acknowledgements.}

Part of this work was done when the author was a postdoc at Memorial University. He has had many fruitful conversations with Professor Jie Xiao. He thanks Professor Xiao for his constant encouragement and support. The author thanks the referee for pointing out Remark \ref{vonneumann}. The author also thanks the referees for their insightful comments which greatly improve the presentation of this paper.

\end{document}